\documentclass[12pt]{amsart}
\usepackage{amssymb}
\usepackage{amscd}
\usepackage{color}
\usepackage{url}
\usepackage{hyperref}

\numberwithin{equation}{section}

\theoremstyle{plain}
\newtheorem{theorem}[subsection]{Theorem}
\newtheorem{proposition}[subsection]{Proposition}
\newtheorem{lemma}[subsection]{Lemma}
\newtheorem{corollary}[subsection]{Corollary}

\newtheorem{rem}[subsection]{Remark}

\theoremstyle{definition}

\newcommand{\rad}{\operatorname{rad}}

\newcommand{\tarc}{\mbox{\large$\frown$}}
\newcommand{\arc}[1]{\stackrel{\tarc}{#1}}
\newcommand{\Z}{\mathbb{Z}}
\newcommand{\R}{\mathbb{R}}
\newcommand{\C}{\mathcal{C}}
\newcommand{\Ga}{\Gamma}

\def\a{\alpha}
\def\b{\beta}
\def\l{\lambda}
\def\r{\rho}
\def\d{\delta}

\begin{document}

\title[Integral points]{Integral points on  convex curves}

\author{Jean-Marc Deshouillers and Adri\'an Ubis}

\address{Institut de Math\'{e}matiques de Bordeaux \\ Universit\'e de Bordeaux \\ Bordeaux INP et CNRS \\ 33405 TALENCE \\ France}
\email{jean-marc.deshouillers@math.u-bordeaux.fr}
\address{Departamento de Matem\'aticas \\ Universidad Aut\'onoma de Madrid \\ Madrid 28049 \\ Spain}
\email{adrian.ubis@uam.es}
\thanks{}
%
%
\begin{abstract}
We estimate the 
maximal number of integral points which can be on a convex arc in $\R^2$ with given length, minimal radius of curvature and initial slope.
\end{abstract}

\maketitle

\hfill \emph{To the memory of Javier Cilleruelo}

\

\section{Introduction}

Evaluating the number of integral points (points with integral coordinates) on finite continuous curves in $\R^2$ is a fairly general Diophantine question. Since the distance between two distinct elements in $\Z^2$ is at least $1$, on a 
simple
curve with length $\ell$ there cannot be more than $\ell + 1$ integral points, a bound which is only achieved for some linear curves.

Besides the study of specific curves, the first general result is due to Jarn\'{i}k \cite{J} who proved in 1925 that the number of points on a strictly convex arc $y=f(x)$ of length $\ell$ is at most $3(4\pi)^{-1/3} \ell^{2/3} + O(\ell^{1/3})$,  and that this bound is reached for some arc. From there Jarn\'ik deduced a similar result for strictly convex simple closed curves, 
giving the optimal bound
$3(2\pi)^{-1/3} \ell^{2/3} + O(\ell^{1/3})$.

In 1963, Andrews \cite{A} gave an upper bound for the number $N$ of integral points on the boundary of a strictly convex body in $\R^n$ in terms of the volume $V$ of that body, which is $N \ll V^{1/3}$ when $n=2$.

Grekos \cite{G}, in 1988, revisited Jarn\'{i}k's method in the case of strictly convex \textit{flat} $\C^2$ curves, \emph{i.e.} curves $\Ga$ for which the length $\ell = \ell(\Ga)$ is smaller than the minimum of 
the
radius of curvature along $\Ga$. Denoting by $r = r(\Ga)$ this minimal radius of curvature and by $N=N(\Ga)$ the number of integral points on $\Ga$, he first obtains the upper bound
\begin{equation}\label{GG}
N \le 2 \ell r^{-1/3}.
\end{equation}
With an unspecified constant, this result can be derived from \cite{A}.

The second result of \cite{G} implies that, up to the constant, (\ref{GG}) is best possible, as long as $\Ga$ is not too flat --- \emph{i.e.} $\log \ell / \log r > 2/3$ --- and the lower bound he obtains for families of curves is uniform in terms of the slope $w = w(\Ga)$ of the curve (\emph{i.e.} the tangent of its angle with the $x$-axis).

The relevance of the slope is pointed out in \cite{DG}: Grekos and the first named author of the present paper showed that for any strictly convex $\C^2$ curve with a tangent at the origin parallel to the $x$-axis, the number of its integral points satisfies
\begin{equation}\label{DG1}
N \le \ell^2/r + \ell/r + 1,
\end{equation}
a quantity which is 
essentially
less than $\ell r^{-1/3}$ when $\log \ell / \log r < 2/3$.

On the other hand, for any $\alpha \in [1/3, 2/3]$ they also constructed curves $\Ga$ for which $N > 0.79 \ell r^{-1/3}$ and $\log \ell / \log r = \alpha$.\\

Our main result expresses how the maximal number of integral points on a \emph{very flat} strictly convex $\C^2$ curve depends on the rational approximation of its slope.  In particular, we show that (\ref{GG}) is essentially best possible for any fixed initial slope in the case $\ell\ge r^{2/3}$, which 
slightly
improves on the result of Grekos ($\ell\ge C_{\epsilon} r^{2/3+\epsilon}$).

Let us first precise our notation. A \emph{strictly convex $\C^2$ curve} $\Ga$ is (the image of) a $\C^2$ map $\gamma = (x, y)$ from $[0, 1]$ to $\R^2$ such that  
$x'y''-x''y'$ never vanishes. Up to an isometry which preserves $\Z^2$ (composition of symmetries with respect to the axes or the main bisectors)  we may assume that $0 \le  y'(0) \le x'(0)$; we then let $w = w(\Ga) = y'(0)/x'(0)$ which belongs to $[0, 1]$. The radius of curvature of $\Ga$ at the point $(x(t), y(t))$ is given by $r(t) = (x'^2+y'^2)^{3/2} / |x'y''-x''y'|$ and we let $r = r(\Ga) = \min_{t \in [0, 1]} r(t)$.  We recall that $\ell = \ell(\Ga)$ denotes the length of $\Ga$ and $N = N(\Ga)$ the number of its integral points. 
We consider curves satisfying $\ell(\Gamma)\le r(\Gamma)$ and notice 
that they are really graphs (or arcs) $y=f(x)$.

A real number $x$ can be decomposed in a unique way as $x=\lfloor x\rfloor + \{x\}$, where $\lfloor x\rfloor$ is a rational integer called the integral part of $x$ and $\{x\}$ is a real number in $[0, 1)$. If $\{x\} \neq 1/2$, there exists a unique integer $[x]$ such that $\|x\|=|x-[x]| < 1/2$;  if $\{x\} =1/2$, we define $[x]$ to be $\lfloor x\rfloor$; in both cases, we call $[x]$ the nearest integer to $x$.

Finally, for functions $f$ and $g\ge 0$, we will also use either $f=O(g)$ or $f\ll g$ as shortcut for $|f|\le Cg$ for some positive constant $C$; $f\asymp g$ meaning both $f\ll g$ and $g\ll f$. 

With those convention and notation, we have

\begin{theorem}[Main result]\label{main_result}
There exist two positive numbers $c_1$ and $c_2$ having the following property: for any $r\ge \ell \ge  1$  and $w \in [0, 1]$, the maximum $N_{w,\ell, r}$ of $N(\Ga)$ where $\Ga$ are curves with $\ell(\Ga)=\ell$, $r(\Ga)=r$ and $w(\Ga)=w$ satisfies
\begin{eqnarray}\label{mainequp}
N_{w,\ell,r} \le c_2\left(1 + \min(\ell r^{-1/3} , \ell \delta_{w,\ell r^{-1}}) \right)
\end{eqnarray}
and
\begin{eqnarray}\label{maineqlow}
c_1\left(1 + \min (\ell r^{-1/3} , \ell \delta_{w,\ell r^{-1}})  \right)\le N_{w,\ell,r},
\end{eqnarray}
with $\delta_{w,x}=\min_{q\in \mathbb N} (q x+\|qw\|)$ for $x>0$.
\end{theorem}

\begin{rem}\label{ellsmall}
In the excluded case $\ell<1$ we trivially have $N_{w,\ell,r}=1$. When $\ell<r^{1/3}$, Theorem \ref{main_result} says that $N_{w,\ell,r}\asymp 1$, so the result does not depend on the slope $w$. The same happens when $\ell>r^{2/3}$, since then the result simply claims that $N_{w,\ell,r} \asymp \ell/r^{1/3}$ (due to the inequality $\delta_{w,x}\ge x$); this is 
a slight
extension of Grekos work \cite{G}, since he proved it for $\ell\gg _{\epsilon} r^{2/3+\epsilon}$. We 
shall
actually show that the result for $\ell =r^{2/3}/12$ implies the case $\ell\ge r^{2/3}/12$. In that sense,  we 
shall
be able to  assume that 
\begin{equation}\label{assum}
1\le \ell \le r^{2/3}/12.
\end{equation}
\end{rem}

\begin{rem}
If $w=0$, we have $q=1$ and then (\ref{mainequp}) reads $N_{0,\ell,r} \le c_2 (1 + \ell^2 r^{-1})$ which is, up to a constant factor,
 the first result of \cite{DG}.
\end{rem}

\begin{rem}
The construction used in \cite{DG} for the above mentioned lower bound consists in considering, for $\alpha \in [1/3, 2/3]$,  curves (selected parts of parabolas) with given $r$,  $\ell = r^{\alpha}$ and $w = r^{-1/3}$.
\end{rem}

\begin{rem}
It may seem curious to restrict the consideration of the slope to one end point of the curve and one may ask what about the other end, or another point of the curve. Indeed, the consideration of the slope $w$ is relevant only when $\ell$ is less than $r^{2/3}$, in which case the curve is extremely flat and the slope of the tangent can be considered as constant over the whole curve. We let the Reader make this point precise.
\end{rem}

Theorem \ref{main_result}  is completely uniform, with $\delta_{w,x}$ measuring how well $w$ can be approximated by rationals with small denominator.  From this result we can derive precise consequences for curves which begin with a fixed rational or irrational slope.

For any $w$ irrational number, we are going to measure its good approximation by rationals by the exponent
\[
 \beta=\beta(w)= \limsup_{j\to\infty} \beta_j
\]
where $(a_j/q_j)_{j\in \mathbb N}$ is the continued fraction of $w$ and $\beta_j$ is defined by the equation
\[
 \left|w-\frac{a_j}{q_j}\right|=q_j^{-\beta_j}.
\]
This is the same as the irrationality exponent defined in \cite[page 168]{B}. If $w$ is a rational number, we define $\beta(w)=\infty$. 
It is known that $2\le \beta\le \infty$ and that in that range the set of real numbers with exponent $\beta$ has Hausdorff measure $2/\beta$ (Jarn\'{i}k-Besicovitch). We are going to show that this exponent determines also the number of  
integral
points on curves with initial slope $w$.

\begin{theorem}[Curves with fixed initial slope]\label{initial_slope}
Let $1/3<\alpha<2/3$. Let $w$ be an irrational number with $\beta(w)=\beta$. Then, we have 
\[
 \limsup_{r\to \infty} \frac{\log N_{w,r^{\alpha},r}}{\log r} =\alpha-\frac 13 
\]
and
\[
 \liminf_{r\to \infty} \frac{\log N_{w,r^{\alpha},r}}{\log r} =\min\left(\alpha-\frac 13, 2\alpha-1+\frac{1-\alpha}{\beta}\right).
\]
\end{theorem}

In the previous result we excluded the ranges $\alpha\le 1/3$ and $\alpha\ge 2/3$ because on them we trivially have
\[
 \lim_{r\to\infty}   \frac{\log N_{w,r^{\alpha},r}}{\log r} = \max\left(0,\alpha-\frac 13\right),              
\]
and in particular the result does not depend on $w$.

\section{Upper bounds}

We use what is defined in the previous section.
We begin by recalling an upper bound obtained in \cite{G} which does not depend on the initial slope. This result 
directly follows from the understanding of
the case $\ell=r^{1/3}$. We are going to give an arithmetic proof based on looking at the slopes between  consecutive integral points. 

\begin{proposition}[Local upper bound] \label{upper_bound_arithmetic}
For  any $r,\ell \ge 1$  we have 
\[
 N_{w,\ell,r} \le 2 \frac{\ell}{r^{1/3}}+2.
\]
\end{proposition}
\begin{proof}
The result is a direct consequence of the fact that a curve with length $\ell=r^{1/3}$ cannot have more than two integral points. Suppose this were not true. We can assume that $0\le w=\tan(\theta_0)\le 1$ and  $x'(t),y'(t)>0$ throughout the curve. The maximal slope that the curve can reach corresponds to the case of an arc of a circle of radius $r$, in which case that slope would be $w_1=\tan(\theta_0+\frac{\ell}{2\pi r})$. Since $\ell/r=1/r^{2/3}\le 1$ and $w\le 1$, we have $w_1\le w+\frac{\ell}{r}$.

Then, if $(x_1,y_1)$, $(x_2,y_2)$ and $(x_3,y_3)$ are three integral points on the curve with $x_1<x_2<x_3$,  we have:
$(a,b)=(x_2,y_2)-(x_1,y_1)$ and $(A,B)=(x_3,y_3)-(x_2,y_2)$ satisfy 
\[
 a+A<\ell,        \qquad    w\le \frac{b}{a}<\frac{B}{A}\le w+\frac{\ell}{r}
 .
\]
Since $a,A,b,B$ are natural numbers, this implies
\[
\frac{1}{\ell^2}< \frac{1}{Aa}\le\frac{B}{A}-\frac{b}{a}<\frac{\ell}{r}
\]
which gives $\ell>r^{1/3}$, a contradiction.
\end{proof}

We now begin the study
of an upper bound  for slopes $w$ that are near to a rational $a/q$ with small $q$. The tools will be geometric in nature.  In the case $w=0$, in \cite{DG} it was shown that one can bound the number of integral points 
on
the curve by the number of horizontal lines $y=n, n\in\mathbb Z$ that touch the curve.  We 
shall
do the same for the case $w=a/q$ with the ``rational'' lines $y=\frac{a}{q}x+\frac{n}{q}$. The following lemma is essentially in \cite{DG}.

\begin{lemma}
Let $3<\ell\le r$ and $\Ga$ as in the previous section. Then $\Ga$ is included in a curvilinear triangle $\mathcal T(A,C,D)$ with: $A=(x(0), y(0))$, $AC$ a straight line with length $\ell$ tangent to $\Ga$ at $A$; $\arc{AD}$ is an arc of a circle with radius $r$ and tangent to $\Ga$ at $A$; $CD$ is orthogonal to $AC$.
\end{lemma}

By using that lemma we are going to prove the following result
\begin{lemma}\label{parallelogram_lemma}
Let $3<\ell\le r/3$. 
Then $\Ga$ is included in a parallelogram with two sides parallel to the $y$ axis, two sides having slope $w$ and 
such that
the size of the sides parallel to the $y$ axis is at most $1.6\ell^2/r$
and its projection over the $x$ axis has length $1.02\ell$.
\end{lemma}
\begin{proof}
Throughout the proof of this lemma, we consider coordinates in the frame $(A,\vec i,\vec j)$, with $A=(x(0), y(0))$ and $\vec i$ (\emph{resp.} $\vec j$) is a unitary vector parallel to the $x$-axis (\emph{resp.} $y$-axis).
Notice that in the previous lemma we have $|CD|=r-\sqrt{r^2-\ell^2}$. 

We begin by showing that $|CD|\le 0.6 \ell^2/r$.
To prove it is equivalent to show $\sqrt{r^2-\ell^2}\le r-0.6 \ell^2/r$. Both $r^2-\ell^2$ and $r-0.6\ell^2/r$ are positive so it is equivalent to $r^2-\ell^2\ge r^2-1.2\ell^2+0.36\ell^4/r^2,$ \emph{i.e.} $0.2\ell^2\ge 0.36 \ell^4/r^2$, or $(\ell/r)^2\le 0.2/0.36$; but we have $(\ell/r)^2 \le 1/9< 0.2/0.36=5/9$.

Since $\ell\le r/3$, we further deduce that $|CD|\le 0.6 \ell^2/r\le 0.2 \ell$ and 
\[
 |AD|=\sqrt{|AC|^2+|CD|^2}\le \ell\sqrt{1.04}\le 1.02 \ell,
\]
so for any point $k=(x_k,y_k)$ in the triangle $ACD$, we have $x_k\le 1.02 \ell$.

On the line $x=1.02 \ell$, we consider the point $P$ which is also on the line $AC$ and $Q$ which is also on the line $AD$. We have the following properties:
\begin{itemize}
 \item[(i)] Any point in $\Ga$ is in the triangle $APQ$.
 \item[(ii)] $|PQ|\le 1.6 \ell^2/r.$
\end{itemize}
The first property is clear. Let us prove the second: if $\varphi$ is the angle between $AC$ and $AD$, since $|CD|\le 0.2 \ell$, we have
\[
 -\arctan(0.2)\le\theta+\varphi\le \frac{\pi}4+ \arctan(0.2)\le 0.983
\]
which implies $\cos(\theta+\varphi)\ge 0.55$
and since $\theta+\varphi$ is the angle between $\vec i$ and $AD$, the horizontal component of the point $D$ satisfies

\[
 x_D\ge 0.55 |AD|\ge 0.55 \ell.
\]

Let $C_1$  be the intersection of the line $AC$ with the line $x=x_D$. Since $\theta\in[0,\frac{\pi}4]$, we have
\[
 |C_1D|=\frac{|CD|}{\cos\theta}\le \sqrt{2} |CD| \le 0.6\sqrt 2 \frac{\ell^2}{r}
\]
and so
\[
 |PQ|=\frac{1.02}{x_D}|C_1 D|\le \frac{1.02}{0.55}  0.6\sqrt 2 \frac{\ell^2}{r} \le 1.6 \frac{\ell^2}{r}.
\]

Now the triangle $APQ$ is contained in the unique parallelogram $APQK$ satisfying the properties in the statement of the lemma.
\end{proof}

We can put the parallelogram from Lemma \ref{parallelogram_lemma} inside one with two sides having rational slope $a/q$ and the other two sides being vertical. This gives the following result.

\begin{proposition}[Curve inside rational parallelogram]\label{inside_rational_parallelogram}
Let $q\ge 1$ and $0\le a\le q$ with $(a,q)=1$. If $3<\ell<r/3$ then $\Ga$ is included in a parallelogram with two sides parallel to $\vec j$, two sides having slope $a/q$ and the size of the sides parallel to $\vec j$ is at most $1.02 \ell|w-a/q|+1.6\ell^2/r$. 
\end{proposition}

Now, it is possible to control the integral points inside such a parallelogram by grouping them onto lines of slope $a/q$, the number of those lines being easy to understand.

\begin{lemma}[Integral points in a rational parallelogram]\label{lines_in_parallelogram} 
Let $q$ and $a$ be coprime integers with $q\ge 1$ and $a\ge 0$,
let $u,v,h,k$ be real numbers with $h>0$ and $k>0$ and let $\mathcal P$ be the parallelogram with vertices $(u,v), (u,v+h), (u+k,v+ak/q), (u+k,v+h+ak/q)$. The number of straight lines with slope $a/q$ which contain at least one integral point from $\mathcal P$ is at most equal to $qh+1$.
\end{lemma}
\begin{proof}
Let $y=ax/q+m$ be the equation of such a straight line. Since it contains at least one integral point, $mq$ is an integer $j$. Since it contains one point in $\mathcal P$, we have $v\le au/q+j/q\le v+h$. Thus the number of straight lines we are counting is at most the number of integers $j$ in the interval $[vq-au,(v+h)q-au]$, whence the result.
\end{proof}

With the two previous results we can finally prove our upper bound for the number of integral points on $\Ga$.

\begin{theorem}[Upper bound for ``rational'' slopes]\label{upper_bound_geometric}
Let $3<\ell\le r/3$ and $\Ga$ with $\ell(\Ga)=\ell$ be such that for any $M\in\Ga$, $r(M)\ge r$. Then for any $q\ge 1$, $a\ge 0$ and $(a,q)=1$, we have
\[
N(\Ga)\le 2.04 q\, \ell\,\left|w-\frac aq\right|+ 3.2 q \frac{\ell^2}{r} +2.
\]
\end{theorem}
\begin{proof}
By  Proposition \ref{inside_rational_parallelogram}, $\Ga$ is contained in a parallelogram $\mathcal P$, and by Lemma \ref{lines_in_parallelogram} the number of lines with slope $a/q$ inside that parallelogram which contain at least one integral point is at most
\[
 q\left(1.02\, \ell \left|w-\frac aq\right|+1.6 \frac{\ell^2}r \right)+1.
\]
Now, each integral point on $\Ga$ is contained on one of those lines. Moreover, since $\Ga$ is strictly convex, each line cannot contain more than two points, so the result follows.
\end{proof}

Notice that by choosing the best $a/q$ possible in Theorem \ref{upper_bound_geometric} we get the bound $N_{w,\ell,r} \ll 1+\ell \delta_{w,\ell r^{-1}}$ from Theorem \ref{main_result} , and considering also Proposition \ref{upper_bound_arithmetic} we have (\ref{mainequp}). In the following section we 
shall
show that those bounds are the only restrictions for $N_{w,\ell,r}$.

\section{Lower bound for ``irrational'' slopes}

The proof of Proposition \ref{upper_bound_arithmetic} shows a relationship between integral points on the curve and rational slopes.  For obtaining lower bounds, both Jarn\'ik \cite{J} and Grekos \cite{G} used Farey fractions as slopes in order to build curves with many integral points. We begin by writing a general result capturing those ideas.

\begin{lemma}[Curve with Farey tangents]
\label{curve_farey}
Let $I$ be an interval contained in $[0,1]$ with $|I|\le 1/30$. Let $M\in\mathbb N$ and $F_M$ be the family of Farey fractions with denominators up to $M$. If $|F_M\cap I|\ge 3$, then there exists a twice differentiable curve $\Gamma\subset \mathbb R^2$ such that
\begin{itemize}
\item[(i)] its length is at most $32 M^3 |I|$,
\item[(ii)] its radius of curvature at each point is in the interval $[\frac 1{16} M^3, 16M^3]$,
\item[(iii)] it has at least $|F_M\cap I|-1$ points with integer coordinates,
\item[(iv)] the slope at its initial point is 
$\frac{h_1+h_2}{k_1+k_2}$,
\end{itemize}
with $h_1/k_1$ and $h_2/k_2$ being the first two terms in $F_M\cap I$.
\end{lemma}
\begin{proof}
Let $I=[s,s+\Delta s]$; let us write the elements of $F_M$ in increasing order
\[
\frac{h_0}{k_0}< s\le \frac{h_1}{k_1}<\frac{h_2}{k_2}< \ldots <\frac{h_N}{k_N}\le s+\Delta s < \frac{h_{N+1}}{k_{N+1}}.
\]
We are going to use those elements to build our curve. We first list $N-1$ points with integer coordinates, which will be on the curve:
\[
 (x_1,y_1)=(0,0)   \; \text{and} \; \forall j \in [2, N-1]\colon (x_j,y_j)=(x_{j-1},y_{j-1})+(\lambda_j k_j, \lambda_j h_j) 
\]
with $\lambda_j=[M^2/k_j^2]$; we recall that $[x]$ denotes the nearest integer to $x$. 
Next, we fix  the slope of the curve at the point $(x_j,y_j)$ to be
\[
\tan \theta_j= \frac{h_j+h_{j+1}}{k_j+k_{j+1}}
\]
for $1\le j \le N$. We are going to use Proposition \ref{Glemma} to make sure that a curve satisfying those requirements and the ones in the statement of the lemma 
does exist.
In fact, in order to build the curve between $A=(x_{j-1},y_{j-1})$ and $B=(x_j,y_j)$, since the line between them has slope $\tan\theta=h_j/k_j$, we check that
\[
 \tan\theta_{j-1}-\tan\theta=\frac{h_{j-1}+h_j}{k_{j-1}+k_j}-\frac{h_j}{k_j}=-\frac{1}{k_j(k_{j-1}+k_j)}
\]
\[
 \tan\theta_j-\tan\theta=\frac{h_{j+1}+h_j}{k_{j+1}+k_j}-\frac{h_j}{k_j}=\frac{1}{k_j(k_{j+1}+k_j)}.
\]
Now, $k_{j+1}+k_j> M$ because otherwise $(h_{j+1}+h_j)/(k_{j+1}+k_j)$ would be a Farey fraction in $F_M$ between $h_j/k_j$ and $h_{j+1}/k_{j+1}$. Thus, $k_{j+1}+k_j\in [M,2M]$ for every $j$, so that
\[
 \frac{\tan\theta_{j-1}-\tan\theta}{ \tan\theta_j-\tan\theta}=-\frac{k_{j+1}+k_j}{k_{j-1}+k_j}\in [-2,-1/2],
\]
and then by applying Lemma \ref{trigonometry} we have
\[
 \frac{\alpha}{\beta}=\frac{\tan(\theta_{j-1}-\theta)}{\tan(\theta_j-\theta)}\in [-3,-1/3].
\]
Also
\[
 \frac{|AB|}{\tan\theta_j-\tan\theta}=\frac{\lambda_j \sqrt{h_j^2+k_j^2}}{1/(k_j(k_{j+1}+k_j))}=\lambda_j k_j^2 (k_{j+1}+k_j) \sqrt{1+(\tan\theta)^2}
\]
so that
\[
 \frac{|AB|}{\tan\theta_j-\tan\theta}\in \left[\frac 12 M^3, 3\sqrt{2}M^3\right],
\]
and then by Lemma \ref{trigonometry} we have
\[
 \frac{|AB|}{\beta}=\frac{|AB|}{\tan(\theta_j-\theta)}\in \left[\frac 13 M^3, 9 M^3\right].
\]
By Proposition \ref{Glemma} we can build a curve between $A$ and $B$ with radius of curvature between $M^3/16$ and $16 M^3$, and joining those pieces the same is true between $(x_1,y_1)$ and $(x_{N-1},y_{N-1})$.

Moreover, by our definition of the curve, in order to finish the proof we just have to show that its length satisfies the condition in the statement of the lemma. But by convexity and considering the slope of the curve we have
\[
  \text{Length}(\Gamma)\le \Delta x+\Delta y\le 2\Delta x.
\]
On the other hand, by the mean value theorem and our control over the curvature of the curve we have
\[
 \Delta s\ge \Delta\left(\frac{dy}{dx}\right)\ge \frac{1}{ 16 M^3} \Delta x
\]
so finally
\[
 \text{Length}(\Gamma)\le 32 M^3 \Delta s
\]
and the result follows.
\end{proof}

In order to take advantage of the previous result we need to control the distribution of Farey fractions in certain intervals.  The question is that for $N_{w,\ell,r}$ we are interested in the Farey fractions near $w$, and that depends on whether $w$ is near a rational with small denominator or not. In the first case, that rational repels other rationals, so we would not have other Farey fractions. In the second
one,
we should have the amount of Farey fractions that would be expected from probabilistic reasoning.

The problem with the analysis in \cite[Lemme 3]{G} and \cite[Corollary 1]{plagne} is that their counting of Farey fractions on an interval $I$ only takes into account its length $|I|$
and does 
not capture the 
subtlety
described in the previous paragraph. We solve that problem with the following result.

\begin{lemma}[Farey fractions in an interval]
\label{farey_in_intervals}
There exists a constant $C>1$ such that for any $a,q$ coprime natural 
numbers with $a/q+1/q^2<1$ and $M$ with $\frac{M}{q}>C$, $z>C $,  the number of Farey 
fractions with denominators up to $M$ in the interval
\[
 \left[\frac aq , \frac aq+\frac{z}{Mq}\right]
\]
is at least $\pi^{-2} z (M/q)$.
\end{lemma}
\begin{rem}
 Notice that for $z<1$ the only possible Farey fraction 
 is
 $a/q$, so there is a sudden change in 
 behaviour
 when $z$ increases
 (especially if $q$ is much smaller than $M$).
\end{rem}

\begin{proof}
We need to count the coprime $h,k$ with $1\le k \le M$ such that
 \[
  0<\frac hk-\frac aq <\frac{z}{Mq}
 \]
which is equivalent to
\[
 0<qh-ak<\frac{z k}{M}.
\]
By restricting to $M/2<k<M$ we see that the number of Farey fractions we want to control is at least
\[
 J = \sum_{\frac M2< k< M} \sum_{m<\frac z2}  \sum_{\substack{qh-ak=m\\(h,k)=1} 
} 1.
\]
We can parametrize the integer solutions of $qh-ak=m$ as
\[
 k=qj-\overline{a}m \qquad h=aj-sm
\]
with $1\le \overline a\le q$ the inverse of $a$ modulo $q$ and $s=\frac{a\overline a-1}{q}$. Also, 
$(h,k)=1$ is equivalent to $(j,m)=1$, and 
thus
\[
 J=\sum_{m<\frac z2} \sum_{\substack{\frac x2<j-\delta m<x\\ 
(j,m)=1}}  1
\]
with $x=M/q$ and $\delta=\overline a/q$. Since $\sum_{d\mid l} \mu(d)=0$ for 
$l>1$ and equals 1 for $l=1$ we have
\[
 J= \sum_{m<\frac z2} \sum_{\substack{\frac x2<j-\delta m<x}}
\sum_{\substack{d\mid j \\ d\mid m}} \mu(d).
\]
By writing $j=j_* d$ and $m=m_* d$ and rearranging the sums we have
\[
 J= \sum_{d<\frac z2} \mu(d) F(d)
\]
with
\[
F(d)=\sum_{m_*<\frac{z}{2d}} 
\sum_{\frac{x}{2d}<j_*-\delta m_*<\frac{x}{d}} 1.
\]
We split the sum as
\[
 J=\sum_{d<\frac y2} \mu(d) F(d) +\sum_{\frac{y}2\le d <\frac z2} \mu(d) F(d)
\]
with $y=\min(z,x)$. We can estimate
\[
 F(d)=\sum_{m_*<\frac{z}{2d}} \left( \frac{x}{2d} +O(1) \right) = \frac{zx}{4d^2}+O\left(\frac 
{z+x}d\right)
\]
so
\[
\sum_{d<\frac y2} \mu(d)F(d)= \sum_{d<\frac y2} \mu(d) \left[ \frac{zx}{4d^2} + 
O\left(\frac{z+x}{d}\right) \right]
\]
so by using $\sum_{d=1}^{\infty} \mu(d)d^{-2}=\zeta(2)^{-1}=6/\pi^2$ we have
\[
 \sum_{d<\frac y2} \mu(d)F(d)=\frac{zx}{4} \frac{6}{\pi^2} 
+O\left(\frac{zx}{y}\right)+O((z+x)\log y)
\]
so for $z,x>C$ with $C$ large enough we have
\[
 \sum_{d<\frac y2} \mu(d)F(d)\ge \frac{zx}{4}\frac{5}{\pi^2}.
\]
To control the other sum let us look at $S_{D,D'}=\sum_{D\le d<D'} \mu(d)F(d)$ 
for any $x\le D\le D'\le 2D$. Rearranging the sums we have
\[
 |S_{D,D'}|\le \sum_{m_*<\frac{z}{2D}} \sum_{\frac{x}{4D}<j_*-\delta 
m_*<\frac{x}D}  \left| \sum_{a_{j_*,m_*}<d<b_{j_*,m_*}} \mu(d) \right|
\]
for some $a_{j_*,m_*},b_{j_*,m_*}$ in the interval $[D,2D]$.  Thus, applying 
the prime number theorem we have
\[
 S_{D,D'}\ll  \sum_{m_*<\frac{z}{2D}} \sum_{\frac{x}{4D}<j_*-\delta 
m_*<\frac{x}D}  \frac{D}{(\log D)^2} \ll \sum_{m_*<\frac{z}{2D}} \frac{D}{(\log 
D)^2} \ll \frac{z}{(\log D)^2}.
\]
Then, by splitting the sum into dyadic intervals we have
\[
 \sum_{\frac{y}{2}\le d <\frac{z}2} \mu(d) F(d)\ll \sum_{\frac{y}{2}\le 2^n
<\frac{z}2} \frac{z}{(\log 2^n)^2} \ll z
\]
so this sum is at most $zx/4\pi^2$ for $x>C$ for $C$ large enough, so 
that we finally get
\[
 J\ge \frac{zx}4 \frac{5}{\pi^2}-\frac{zx}4 \frac 1{\pi^2}=\frac{zx}{\pi^2}.
\]

\end{proof}

Now we are going to build curves with many integral points by using Farey fractions. But we are going to do it just in the ``irrational'' case, namely when $w$ is not near to a rational with small denominator. Afterwards we 
shall
see that in the other case, the ``rational'' case, we 
shall
need 
other tools.

\begin{theorem}[Lower bound for ``irrational'' slopes]
\label{lower_bound_arithmetic}
There exists a constant $C>1$ such that:
for every $r, \ell>1$ with $800C^2 r^{1/3}<\ell<r^{2/3}$ and  $w\in (0, 1)$, if there is no rational $a/q$ with $q\le 800C^4 r^{2/3}/\ell$ and $|w-a/q|\le 1/qr^{1/3}$ then 
 we can build a curve $\Gamma$ in $\mathbb R^2$ 
satisfying the following properties:
\begin{itemize}
\item[(i)] $\Gamma$ is twice differentiable and its radius of curvature is 
always in the range $[C^3r/32, 32C^3r]$,
\item[(ii)] the length of $\Gamma$ is less than $\ell$,
\item[(iii)] the initial slope of $\Gamma$ is $w$,
\item[(iv)] $|\Gamma\cap \mathbb Z^2|\ge \frac{1}{3200\pi^2C} \ell r^{-1/3}$.
\end{itemize}
\end{theorem}
\begin{proof}
Dirichlet lemma tells us that there is always an irreducible rational $a/q$ with $q<r^{1/3}$ such that $|w-\frac aq|\le \frac{1}{qr^{1/3}}$, and by our hypotheses we can assume that $q> 800C^4 r^{2/3}/\ell$.

We shall be able to choose $C$ as the maximum of $8\pi^2$ and the constant in 
the statement of Lemma \ref{farey_in_intervals}. Pick $M=C[r^{1/3}]$. By 
applying Lemma \ref{curve_farey} with the interval $I=[a/q,a/q+\ell/400C^3r]$ and 
by Lemma \ref{farey_in_intervals}, since $M>Cq$ and $z=\frac{\ell 
M}{r}\frac{q}{400C^3}> C$ we can build a curve $\tilde \Gamma$ with at least
\[
 \frac{z M/q}{\pi^2}-1\ge \frac{z M/q}{2\pi^2}=\frac{1}{2\pi^2} 
\frac{\ell}{400C^3r}M^2\ge \frac{1}{1600\pi^2 C}  \frac{\ell}{r^{1/3}}
\]
points with integer coordinates, with radius of curvature always in the range
$[M^3/16,16M^3]$, length at most $32M^3 \ell/400C^3r\le \frac 34 \ell$ and with 
initial slope
\[
 \tan\theta=\frac{a+h_2}{q+k_2}.
\]
In order to finish building our curve $\Gamma$, we only need to fix the problem 
that the initial slope should be $w=\tan\theta_w$ instead of $\tan\theta$, but
\[
  |\tan\theta-\tan\theta_w|\le 
\left|\frac{a+h_2}{q+k_2}-\frac{a}{q}\right|+\left|\frac{a}{q}-w\right|\le 
\frac{1}{qM}+\frac{1}{qr^{1/3}}\le \frac{1}{400C^4}\frac{\ell}{r}.
\]
Then, if $w\le\frac{a+h_2}{q+k_2}$, by prolonging $\tilde\Gamma$ to the left 
maintaining the curvature constant we can make sure that the initial slope is 
$w$ and the length of $\Gamma$ will be the length of $\tilde\Gamma$ plus at most 
the length of an arc of a circle of radius $16M^3$ and angle $\theta-\theta_w$, 
namely
\[
 \text{Length}(\Gamma)\le \frac{3}{4} \ell + (\theta-\theta_w)(16M^3)\le 
\frac{3}{4} \ell + \frac{1}{400C^4} \frac{\ell}{r}16 M^3\le  \ell,
\]
by the inequality $\theta-\theta_w\le \tan\theta-\tan\theta_w$. On the other 
hand, in the case $w>\frac{a+h_2}{q+k_2}$, we shall delete the first part of 
$\tilde \Gamma$ to get $\Gamma$, precisely until the slope is $w$. In doing so, 
we delete some of the points with integer coordinates belonging to $\tilde 
\Gamma$.  By construction, the number of points 
which
we counted previously 
and we are deleting now is at most 1 plus the number of Farey fractions with denominator up to 
$M$ in the interval
\[
 \left[\frac aq, w\right].
\]
Since the distance between two consecutive Farey fractions is at least 
$M^{-2}$, the number of them in that interval is at most
\[
 1+ M^2 |w-\frac aq|\le 1+ M^2 \frac{1}{qr^{1/3}}\le   \frac{1}{400 
C^2} \ell r^{-1/3}.
\]
Thus, the number of 
``surviving" integral points
is at least
\[
 \frac{1}{1600\pi^2 C}  \frac{l}{r^{1/3}}- \frac{1}{400 
C^2} \ell r^{-1/3}\ge \frac{1}{3200\pi^2 C}  \frac{\ell}{r^{1/3}}.
\]

\end{proof}

\section{Lower bound for ``rational'' slopes}

Now we are going to handle the ``rational'' case, that is when $w$ is near to a rational $a/q$ with small $q$. In this case, the method of Farey fractions stops working when $\ell<r^{2/3-\epsilon}$ , since then there will not be any rational with denominator up to $r^{1/3}$ in the interval $(\frac aq, \frac aq+\frac{\ell}{r})$. 

We know from the proof of Theorem \ref{upper_bound_geometric} that in this case we can essentially control the amount of points in the curve by the number of lines $y=\frac{a}{q}x+\frac{n}{q}$, $n\in \mathbb Z$ that touch the curve (since by convexity the curve cannot have 
more than
two points on a line).

Then, in order to build a curve with many integral points, we 
shall
explicitly choose a sequence of  integral points $((x_n,y_n))_{0\le n\le  N}$, with $(x_n,y_n)$ belonging to the the line $y=\frac aq x+\frac{n}{q}$ such that they can be put on a curve with the 
requested
curvature, slope and length.

From the point of view of slopes, we 
can
say that the fractions relevant to the problem can be parametrized in terms $a,q$. We shall split the proof into two 
parts:
the first one will work for $w$ very near $a/q$ (with $q$ small), and the second one for $w$ near to $a/q$, but not that much.

\begin{theorem}[First lower bound for ``rational''  slopes]\label{firstlower}
\label{lower_bound_very_near}
Let $C>1$, $r, \ell>1$ with $(800C^4)^2 r^{1/3}<\ell<r^{2/3}$,
$w \in (0, 1)$ and $|w-a/q|\le \ell/25r$ for some irreducible rational $a/q$ with 
$q<800C^4 r^{2/3}/\ell$, we can build a curve $\Gamma$ in $\mathbb R^2$ 
satisfying the following properties:
\begin{itemize}
\item[(i)] $\Gamma$ is twice differentiable and its radius of curvature is 
always in the range $[r/16, 16r]$,
\item[(ii)] the length of $\Gamma$ is less than 
$\ell$,
\item[(iii)] the initial slope of $\Gamma$ is $w$,
\item[(iv)] $|\Gamma\cap \mathbb Z^2|\ge \frac 14+\frac 14 \frac 1{(800 C^4)^6} q\ell^2/r$.
\end{itemize}
\end{theorem}
\begin{proof}
We can assume that  $q\ell^2/r\ge (800 C^4)^6$ since otherwise we just need to build a curve with at least one integral point, which is trivial. Let us suppose that $w\ge a/q$, the other case being similar. Define $\Omega=q[kr/q^2 \ell]$ with $k=(800C^4)^2$  and consider the sequence $(x_j,y_j)_{0\le  j\le N}$ with $x_0=0,y_0=0$,
\[
 \Delta x_j=\Omega -\overline{a} -(j-1) q[\Omega^3/r]
\]
\[
 \Delta y_j=\frac aq \Delta x_j +\frac 1q,
\]
with $\Delta b_j=b_j-b_{j-1}$,  $\overline a$  the number between $1$ and $q$ which is the inverse of $a$ modulo $q$ and $N=[(1/k) r/q\Omega^2]$. The definition implies that $x_j,y_j\in \mathbb Z$. Also, our hypotheses imply that $\Omega>800Cq$, $\Omega^3/r\ge 1$ so that
\begin{equation}\label{deltaxjbound}
 \Omega\left(1-\frac{1}{400C}\right)<\Omega-q-Nq\Omega^3/r<\Delta x_j<\Omega
\end{equation}
and in particular $\Delta x_j>0$.

We are going to see that it is possible to build a curve $\Gamma$ 
which satisfies
the conditions of the statement 
and
contains the previous sequence
of points.
First we fix the slope of $\Gamma$ at the point $(x_j,y_j)$, with $1\le j\le N-1$, to be
\[
 \tan \theta_j=\frac{\Delta y_{j+1}+\Delta y_j}{\Delta x_{j+1}+\Delta x_j}=\frac aq +\frac 1q \frac{2}{\Delta x_{j+1}+\Delta x_j}.
\]
Now, to see that a curve satisfying the curvature condition in the statement 
does exist, it is enough to see that we can apply Proposition \ref{Glemma} with 
$A=(x_{j-1},y_{j-1})$, $B=(x_{j},y_{j})$, $T_A$ and $T_B$  lines with slopes 
$\tan\theta_{j-1}$ and $\tan\theta_{j}$ respectively, and $\rho=r$. Since the 
line between $A$ and $B$ has slope
\[
 \tan\theta=\frac{\Delta y_j}{\Delta x_j}=\frac aq +\frac 1q \frac{1}{\Delta x_j}
\]
by using (\ref{deltaxjbound}) we have
\[
 \tan\theta_{j-1}-\tan\theta=\frac 1q\frac{\Delta(\Delta x_{j})}{\Delta x_j (\Delta x_{j-1}+\Delta x_j)}=-\frac{[\Omega^3/r]}{2\Omega^2(1-\epsilon)^2},
\]
\[
 \tan\theta_j-\tan\theta=-\frac 1q\frac{\Delta(\Delta x_{j+1})}{\Delta x_j (\Delta x_{j+1}+\Delta x_j)}=\frac{[\Omega^3/r]}{2\Omega^2(1-\epsilon')^2},
\]
with $0<\epsilon,\epsilon'<1/400C$, so
\[
 \frac{\tan\theta_{j-1}-\tan\theta}{\tan\theta_j-\tan\theta}\in [-1/(1-1/400C)^2,-(1-1/400C)^2]
\]
and since $C\ge 1$
, 
by Lemma \ref{trigonometry}
,
 we have
\[
 \frac{\alpha}{\beta}=\frac{\tan(\theta_{j-1}-\theta)}{\tan(\theta_{j}-\theta)} \in [-3,-1/3].
\]
Moreover $|AB|=s\Delta x_j$ with $1<s<2$ and $\beta=\tan(\theta_j-\theta)=t(\tan\theta_j-\tan\theta)$ with $1/2<t<3$, hence
\[
 \frac{|AB|}{\beta \rho}=\frac{s\Delta x_j}{rt(\tan\theta_j-\tan\theta)} =\frac{s}{t}\frac{\Omega(1-\epsilon)}{r[\Omega^3/r]/2\Omega^2(1-\epsilon')^2}=
\]
\[ 
 =\frac{2s}{t} (1-\epsilon)(1-\epsilon')^2\frac{\Omega^3/r}{[\Omega^3/r]}\in [1/3,9],
\]
for some $0<\epsilon,\epsilon'<1/400C$,
so indeed we can build a suitable curve between $A$ and $B$.
By convexity
,
the length of $\Gamma$ is at most
\[
 \sum_{j\le N} (\Delta x_j+\Delta y_j)\le \sum_{j\le N} 3\Delta x_j \le 3(\frac 1k \frac r{q\Omega^2})\Omega\le \frac{6}{k^2}l\le \frac{\ell}{800}.
\]
Then, the curve $\Gamma$ we have just built, beginning 
at
the point $(x_1,y_1)$, 
satisfies all the requirements but 
not necessarily
the one about the initial slope. By 
(\ref{deltaxjbound}) we have
\[
 \tan\theta_1-\frac{a}{q}=\frac 1q \frac{2}{\Delta x_2 +\Delta x_1}\le \frac 1q \frac{1}{(1-1/400C)\Omega}\le \frac{4}{k}\frac{\ell}{r} \le \frac{\ell/r}{800}.
\]
so the condition $w-a/q<\ell/25r$ implies that
\[
 |\tan \theta_1-w|<\frac 1{20}\frac{\ell}{r}.
\]
Thus, as in the proof of Theorem \ref{lower_bound_arithmetic}, it is enough to 
enlarge the curve $\Gamma$ by an arc of a circle of radius at most $16r$ and 
angle at most $\ell/20r$. Then, the expanded curve will satisfy the initial 
condition and have length at most
\[
 \frac{\ell}{800}+16r\left(\frac 1{20}\frac lr\right)< \ell.
\]
\end{proof}

\begin{theorem}[Second lower bound for ``rational'' slopes]\label{lower_bound_near}
Let $C>1$,  $r, \ell>1$ with $(800C^4)^2 r^{1/3}<\ell<r^{2/3}$ and  $w=\tan\theta_w$ with 
$0<w<1$ and $\ell/25r<|w-a/q|\le 1/qr^{1/3} $ for some irreducible rational $a/q$ with 
$q<800C^4 r^{2/3}/\ell$.  We can build a curve $\Gamma$ in $\mathbb R^2$ 
satisfying the first three properties in Theorem \ref{lower_bound_very_near} and
\[
|\Gamma\cap \mathbb Z^2|\ge \frac 14+ \frac 14 \frac{1}{(800C^4)^6}q\ell  \left|w-\frac aq \right|.
\]
\end{theorem}
\begin{proof}
We can assume that $q\ell |w-a/q|\ge (800C^4)^6$. We also suppose that $w\ge a/q$. We proceed to build the sequence and the curve as in the proof of Theorem \ref{lower_bound_very_near}
, but define
\[
 \Omega=q\left\lfloor\frac{1}{q(qw-a)}\right\rfloor.
\]
By our hypotheses we again have $\Omega>800Cq$ and $\Omega^3/r\ge 1$, so that the curvature condition is satisfied (since those inequalities are the only thing about $\Omega$ that we used to get the curvature condition in the proof of Theorem \ref{lower_bound_very_near}). Now, we are going to keep just the part of the curve 
which
has the first 
\[
 M=\left\lfloor\frac{1}{(800C^4)^6} q\ell \left|w-\frac aq\right|\right\rfloor
\]
points from the sequence. This is possible since $M\le N=[(1/k)r/q\Omega^2]$
occurs whenever
\[
 \left|w-\frac aq\right|\ge \frac{16}{k^2} \frac {\ell}{r},
\]
which is true due to our 
hypothesis
$|w-a/q|\ge \ell/25r$. Thus, this curve also has the desired number of integral points. Moreover, its length is at most
 \[
 \sum_{j\le M} (\Delta x_j+\Delta y_j)\le \sum_{j\le M} 3\Delta x_j \le 3M\Omega\le \frac{\ell}{800}.
\]
It remains to see what happens with the initial slope. We have 
\[
 \Delta x_1=\Omega-\overline a=\frac{1}{qw-a}-\delta q
\]
for some $0<\delta<2$, hence the slope of the line between $(x_0,y_0)$ and $(x_1,y_1)$ is
\[
 \tan\theta=\frac aq +\frac 1q \frac{1}{\frac{1}{qw-a}-\delta q}=\frac aq+ \frac{w-a/q}{1-\delta q(qw-a)}=w+\frac{\delta(qw-a)^2}{1-\delta q(qw-a)}
\]
and since $|q(qw-a)|\le 1/4$ we have
\[
 0<\tan\theta-w<4(qw-a)^2\le \frac{4}{r^{2/3}}\le \frac{1}{1600} \frac{\ell}{r}.
\]
Moreover
\[
 0<\tan\theta_1-\tan\theta\le \frac{\Omega}{r}\le\frac{1}{r(qw-a)}\le \frac 1{1600}  \frac{\ell}{r}
\]
since we assumed $q\ell |w-a/q|\ge (800C^4)^6$. Thus, finally we have
\[
 0<\tan\theta_1-w<\frac 1{800}  \frac{\ell}{r}
\]
so we can finish as in the proof of 
Theorem \ref{firstlower}
by enlarging the curve to the left of $(x_1,y_1)$ with an arc of a circle.

\end{proof}

\section{Proofs of the main results}\label{se:proofs}

We begin by proving Theorem \ref{main_result}. First we are going to show that, as mentioned in Remark \ref{ellsmall}, we can restrict ourselves to the case $\ell\le r^{2/3}/12$. 

If $\ell\gg r^{2/3}$, Theorem \ref{main_result} says that $N_{w,\ell, r}\asymp \ell r^{-1/3}$. Thus, by cutting a curve of length $\ell>r^{2/3}/12$ into pieces of length between $r^{2/3}/24$ and $r^{2/3}/12$ we see that applying the bound $N_{w,\ell,r}\ll \ell r^{-1/3}$ for $\ell\le r^{2/3}/12$ implies the same bound for $\ell>r^{2/3}/12$.
Regarding the lower bound, we can assume first that $r$ is sufficiently large, since in the case $\ell\le r  \ll 1$ trivially $N_{w,\ell,r}\asymp 1$. We begin by building a curve $\Gamma_1$ with initial slope $w\le 1$, length $(500r)^{2/3}/24$ and radius of curvature larger than $500r$, with $\gg r^{1/3}$ integral points.  If this curve ends at a point $A$ with slope $\tan\theta_1$, we build another curve $\Gamma_2$ with the same conditions but initial slope $\tan(\theta_1+2r^{-1/3})$.
Consider the point $\tilde B$ such that the line passing through $A$ and $\tilde B$ has slope $\tan(\theta_1+r^{-1/3})$ and the distance from $A$ to $\tilde B$ equals $r^{2/3}$.
By an integral translation, we can assume that the initial
point $B$ of $\Gamma_2$ is at distance at most 1 from $\tilde B$. This implies that $|AB|=r^{2/3}+O(1)\sim r^{2/3}$, $\tan(AB,T_A)=-r^{-1/3}+O(\frac{1}{r^{2/3}+O(1)})\sim -r^{-1/3}$ with $T_A$ the line tangent to $\Gamma_1$ at $A$, and $\tan(AB, T_B)\sim r^{-1/3}$ with $T_B$ the line tangent to $\Gamma_2$ at $B$. 
Then, we can apply Proposition \ref{Glemma} with $\rho=250r$ to join $\Gamma_1$ to $\Gamma_2$ so that the full curve is $\mathcal {C}^2$ and of length $O(r^{2/3})$. We continue this procedure with curves $\Gamma_1,\Gamma_2,\Gamma_3,\ldots$, joining $\Gamma_{i}$ to $\Gamma_{i+1}$ until we get a $\mathcal C^2$ curve $\Gamma$ with length at most $\ell$, radius of curvature at least $r$ and $\gg (\ell/r^{2/3})r^{1/3}=\ell r^{-1/3}$ integral points.

After the last paragraph, we can  assume $\ell\le r^{2/3}/12$.  
Let $q_0$ be a
natural number for which the minimum $\delta_{w,\ell r^{-1}}=\min_{q\in\mathbb N} (q\ell r^{-1}+\|qw\|)$ is reached, and let $a_0=[q_0 w]$. This implies $(a_0,q_0)=1$, and applying Proposition \ref{upper_bound_arithmetic} and Theorem \ref{upper_bound_geometric} with $a_0$ and $q_0$ we have
\begin{equation}\label{equation_upper}
 N_{w,\ell,r}\ll 1+\min (\ell/r^{1/3}, \ell \delta_{w,\ell r^{-1}}).
\end{equation}
For $l\ll r^{1/3}$ this implies $N_{w,\ell,r}\ll 1$. In this range it is trivial to build a curve satisfying the curvature condition and with at least one integral point, so that  $N_{w,\ell,r}\asymp 1$. 

Then we can assume $Kr^{1/3}<\ell\le r^{2/3}/12$ for any fixed constant $K$. Then we can apply either Theorem \ref{lower_bound_arithmetic}  or Theorem \ref{lower_bound_very_near} or Theorem \ref{lower_bound_near} in order
to get the bound
\[
 N_{w,\ell,r}\gg 1+\min (\ell/r^{1/3}, \ell(\|qw\|+q\ell/r)).
\]
for some $q\in \mathbb N$, which clearly implies
\[
 N_{w,\ell,r}\gg 1+\min(\ell r^{-1/3}, \ell \delta_{w,\ell r^{-1}})
\]
so from this and (\ref{equation_upper}) we deduce Theorem \ref{main_result}.

Now we are going to derive Theorem \ref{initial_slope} from Theorem \ref{main_result}. For $w$ rational this deduction is trivial. For $w$ irrational, 
let us pick any $r$ sufficiently large, and choose the pair of convergents $q_j,q_{j+1}$ of $w$ such that
\begin{equation}\label{fractions_range}
 q_j \le r^{1/3} < q_{j+1}.
\end{equation}
In this range
\[
\frac{1}{2q_jq_{j+1}}\le  \left|w-\frac{a_j}{q_j}\right|\le \frac{1}{q_jq_{j+1}}\le \frac{1}{q_j r^{1/3}}
\]
and we can show that 
\[
 \min(r^{-1/3},\delta_{w,\ell r^{-1}})\le \min (r^{-1/3}, \|q_j w\|+q_j\ell r^{-1}) \le 10 \min(r^{-1/3},\delta_{w,\ell r^{-1}}) 
\]
with $\ell=r^{\alpha}$. The first inequality comes from the definition of $\delta_{w,\ell r^{-1}}$; if the second were not true we would have $\delta_{w,\ell r^{-1}}<\frac{r^{-1/3}}{10}$ and then  $q<\frac{r^{1/3}}{10}, \|qw\|<\frac{r^{-1/3}}{10}$ for the $q$ such that $\delta_{w,\ell r^{-1}}=\|qw\|+q\ell r^{-1}$, and this would 
contradict
the inequality $(qq_j)^{-1}	\le |w-\frac aq|+|w-\frac{a_j}{q_j}|$, $a=[qw]$.

Then, by Theorem \ref{main_result} we have
\[
 N=N_{w,r^{\alpha},r}\asymp \min\left(r^{\alpha-1/3}, q_j r^{2\alpha-1}+r^{\alpha}/q_j^{\beta_j-1}\right)
\]
since $q_{j+1}\asymp q_j^{\beta_j-1}$. This also gives the inequality 
$
q_j\le r^{1/3}\ll q_j^{\beta_j-1}.  
$
By choosing $r^{1/3}=q_j$ we have $q_jr^{2\alpha-1}=(r^{\alpha-1/3})^2\ge r^{\alpha-1/3}$ so $N\asymp r^{\alpha-1/3}$ and the result for the upper limit follows. On the other hand, any $r$ in the range (\ref{fractions_range}) can be written as $r\asymp q_j^{3\theta}$ with $1\le \theta\le \beta_j-1$ or%
, in other terms,
$q_j\asymp r^{\epsilon}$ with $1/3(\beta_j-1)\le \epsilon\le 1/3$. Then
\[
 N\asymp \min (r^{\alpha-1/3}, r^{\epsilon+2\alpha-1}+r^{\alpha-\epsilon(\beta_j-1)} ).
\]
Now, we want to 
compute
the minimum of this function $N=N(\epsilon)$ in the interval $1/3(\beta_j-1)\le \epsilon\le 1/3$. Since $r^{\alpha-1/3}$ is constant in $\epsilon$, we only need to look at the minimum of the second term. Moreover, both at $\epsilon=1/3(\beta_j-1)$ and $\epsilon=1/3$ we have $N(\epsilon)\asymp r^{\alpha-1/3}$, so we only need to check the case in which the second term has a minimum in the interior of the interval, and this happens at $\epsilon=(1-\alpha)/\beta_j$ whenever
\[
 \frac{1}{3(\beta_j-1)}< \frac{1-\alpha}{\beta_j}<\frac{1}{3}.
\]
The second inequality is always true, but the first amounts to
\[
 \beta_j>1+\frac{1}{2-3\alpha}.
\]
In this case, the minimum for the second term is $\asymp r^{2\alpha-1+\frac{1-\alpha}{\beta_j}}$. 
This implies the result for the lower limit.

\section{Appendix}

\begin{lemma}[Trigonometric lemma]
\label{trigonometry}
Let $\tan\theta$ and $\tan(\theta+\tilde \beta)$ be in the interval $[s, s+\Delta s]\subset [0,1]$, with $0\le \Delta s<1/2$. Then we have
\[
 1-\Delta s\le \frac{\tan(\theta+\tilde\beta)-\tan\theta}{(1+(\tan\theta)^2)\tan\tilde\beta}\le \frac{1-\Delta s}{1-2\Delta s}.
\]
\end{lemma}
\begin{proof}
The formula for the tangent of the sum of two angles gives
\[
 \tan(\theta+\tilde\beta)-\tan\theta= \tan\tilde\beta \frac{1+(\tan\theta)^2}{1-\tan\theta\tan\tilde\beta},
\]
which implies
\[
 |\tan\tilde\beta|\le |\tan(\theta+\tilde\beta)-\tan\theta| (1+|\tan\tilde\beta|)\le \Delta s (1+|\tan\tilde\beta)
\]
so that $|\tan\tilde\beta|\le \Delta s/(1-\Delta s)$. Substituting this bound into the previous identity 
ends
the proof.
\end{proof}

The following result is essentially a variation on 
construction in \cite{G}.

\begin{proposition}\label{Glemma}
Let $A$ and $B$ be two points in the euclidean plane $E = \mathbb{R}^2$, $T_A$ (\emph{resp}. $T_B$) be a straight line containing $A$ (\emph{resp}. $B$) and $\a, \b, \r_1, \r_2, \r$ be real numbers such that
\begin{eqnarray}
\notag
&(\text{\textrm{i}})& \quad \tan(AB, T_A) = \a \text{ and } \tan(AB, T_B) = \b,\\
\notag
&(\text{\textrm{ii}})& \quad \b \in (0, 1/3], \a \in [-3 \b, -\b/3] \textrm{ and } 0<\r\le \min(\r_1,\r_2),\\
\notag
&(\text{\textrm{iii}})& \quad |AB| \in [(1/3) \b \r\; , \; 9 \b \r].
\end{eqnarray}
There exists a two times differentiable curve with end points $A$ and $B$, which admits for tangent at the point $A$ (\emph{resp}. $B$) the line $T_A$ (\emph{resp}. $T_B$), 
such that its radius of curvature is always between $\r/250$ and $250\max(\r_1,\r_2)$ and 
which
takes the value $(1+\a^2)^{3/2}\rho_1$ (\emph{resp}. $(1+\b^2)^{3/2}\r_2)$ at the point $A$(\emph{resp}. $B$).
\end{proposition}

We first prove a technical Lemma, in the spirit of Lemma 1 in \cite{G}.

\begin{lemma}\label{GGaux}
Let $\r>0, a < b, \b >0, \a = -\l\b$ with $\l \in [1/3, 3]$ be real numbers such that
\begin{equation}\label{ddelta}
(1/3) \b \r \le b - a \le 9 \b \r.
\end{equation}
There exists a differentiable real function $f$ defined on $[a, b]$ such that
\begin{eqnarray}
\notag
&(\text{\textrm{i}})& \quad f(a) = \a  \; \text{ and }\; f(b) = \b,\\
\notag
&(\text{\textrm{ii}})& \quad f'(a) = 1/\r_1, \, f'(b)=1/\r_2,\\
\notag
&(\text{\textrm{iii}})& \quad \forall x \in [a, b] \; : \; 0.01/\max(\r_1,\r_2) \le f'(x) \le 100/\r, \\
\notag
&(\text{\textrm{iv}})& \quad \int_a^b f(t) dt = 0.
\end{eqnarray}
\end{lemma}

\begin{proof}
We consider the points $M_A = (a, \a)$ and $M_B = (b, \b)$. The slope of 
the
segment $M_A M_B$ lies in the interval $[0.3/\r, 12/\r]$: we have indeed
$$
\frac{0.3}{\r}\le\frac{4}{9\r} \le \frac{(4/3)\b}{3\b\r} \le \frac{\b(1+\l)}{b-a} = \frac{\b - \a}{b-a} = \frac{\b(1+\l)}{b-a} \le \frac{4\b}{\b \r /3} = \frac{12}{\r}.
$$
This implies that the straight lines $D_A$ passing through $M_A$ and having the slope $0.01/\r < 0.3/\r$ and the lines $D_B$
passing through $M_B$ and having the slope $100/\r > 12/\r$ will meet at a point $M_1 = (x_1, y_1)$ with $a < x_1 < b$ and $\a < y_1 < \b$. \\
We now consider the function $h_1$ defined on $[a, b]$, linear on $[a, x_1]$ and on $[x_1, b]$ and which takes the values: $h_1(a) = \a, h_1(x_1) = y_1$ and $h_1(b)=\b$. Let us show that
$$
I_1 = \int_a^b h_1 (x) dx  < 0.
$$
We have 
$$
2 I_1 = (x_1 - a) (\a+y_1)+  (b - x_1)(\b + y_1);
$$
The coordinates of the point $M_1$ are defined by
\begin{equation}\label{M1}
\frac{y_1-\a}{x_1-a}= 0.01/ \r \quad \text{ and }\quad \frac{\b-y_1}{b-x_1}=100 /\r,
\end{equation}
from which we get
\begin{eqnarray}\label{I11}
2 I_1 &=& \r \left\{100 (y_1^2-\a^2)+0.01(\b^2-y_1^2)\right\} \\
&=& \r \b^2 \left\{99.99 \left(\frac{y_1}{\b}\right)^2 + (0.01-100\l^2)\right\}.
\end{eqnarray}

From (\ref{M1}) we can compute $y_1$ and get
$$
\frac{1}{\r} = 100 (y_1 - \a) + 0.01 (\b y_1)
$$
which leads to
\begin{eqnarray}
\notag
\frac{b-a}{\r \b} &=& 100 \left(\frac{y_1}{\b} - \frac{\a}{\b}\right) + 0.01 \left(1 - \frac{y_1}{\b}\right)\\
\notag
&=& 99.99 \left(\frac{y_1}{\b}\right) + (100 \l + 0.01).
\end{eqnarray}
Since $(b-a)/(\r \b) \in [1/3, 9]$ and $\l \in [1/3, 3]$, we have
$$
-99.893 \l \le -100 \l + 1/3 -0.01 \le 99.99 (y_1/\b) \le 9 - (100 \l +0.01) <0,
$$ 
and so 
$$
\left(\frac{y_1}{\b}\right)^2 \le 0.9981 \l^2.
$$
We 
incorporate
this last relation in (\ref{I11}) and get
$$
2I_1 /( \r \b^2) \le 99.99\times 0.9981 \l^2 + 0.01 - 100 \l^2 < 0.01 - 0.19 \l^2 < -0.011,
$$
which proves that $I_1$ is negative.\\
For $\d$ positive and sufficiently small, we also have
$ \int_a^b (h_1 (x)  + 2\d )dx  < 0$. We can then slightly modify the function $h_1 + \d$ to get a function $f_1(x)$ which satisfies the conditions $(i), (ii)$ and $(iii)$ of Lemma \ref{GGaux} and such that $J_1 = \int_a^b f_1(x) dx < 0$.\\

In a similar way, considering first the straight lines $\Delta_A$ (\emph{resp.} $\Delta_B$) passing through $M_A$ (\emph{resp.} $M_B$) and having the slope $100/\r$ (\emph{resp.} $0.01/\r$), one can construct a function $f_2(x)$ which satisfies the conditions $(i), (ii)$ and $(iii)$ of Lemma \ref{GGaux} and such that $J_2 = \int_a^b f_2(x) dx >0$.\\

The function $f$ defined by $f(x) = (J_2 f_1(x) -J_1 f_2(x))/(J_2-J_1)$ satisfies the conditions $(i), (ii), (iii)$ and $(iv)$ of the Lemma \ref{GGaux}.

\end{proof}

For a two times differentiable function $g$, we denote by $\rad_g(x)$ its radius of curvature at the point $x$, given by
\begin{equation}\label{radius}
\rad_g(x) = \frac{\left(1+g'(x)^2\right)^{3/2}}{|g''(x)|}.
\end{equation}

\begin{corollary}\label{Gcurv}
Let $\min(\r_1,\r_2)\ge \r>0, a < b, \b \in (0, 1/3], \a = -\l\b$ with $\l \in [1/3, 3]$ be real numbers such that
\begin{equation}\label{ddelta}
1/3 \b \r \le b - a \le 9 \b \r.
\end{equation}
There exists a two times differentiable real function $F$ defined on $[a, b]$ such that
\begin{eqnarray}
\notag
&(\text{\textrm{i}})& \quad F(a)=F(b)=0,\\
\label{Fprimenab}
\notag
&(\text{\textrm{ii}})& \quad F'(a)=\a, F'(b) =\b,\\
\label{Frdab}
\notag
&(\text{\textrm{iii}})& \quad \text{We have } \rad_F(a) = \r_1 (1+\a^2)^{3/2} \text{ and } \rad_F(b) = \r_2 (1+\b^2)^{3/2}, \\
\label{Frad}
\notag
&(\text{\textrm{iv}})& \quad \forall x \in [a, b] \; \colon \;  \rad_F(x) \in [0.01\r, 300 \max(\rho_1,\rho_2)].
\end{eqnarray}
\end{corollary}

\begin{proof}
For the given parameters $\r,\rho_1,\rho_2, a, b, \b, \l$, we construct a function $f$ satisfying the conditions of Lemma \ref{GGaux}. We define
\begin{equation}\notag
\forall x \in [a, b] \; : \;  F(x) = \int_a^x f(t) dt.
\end{equation}
Relations $(i), (ii)$ and $(iii)$ come directly from Lemma \ref{GGaux} and (\ref{radius}). Relation $(ii)$ of Lemma \ref{GGaux} implies that for all $x$ one has $0.01\r \le 1/F''(x) \le 100\max(\r_1,\r_2)$; so $F'$ is increasing and then  $|F'(x)| \le \max(|\a|, \b) \le 1$; relation $(iv)$ easily follows from those relations and (\ref{radius}).
\end{proof}

\emph{Proof of Proposition \ref{Glemma}.}
Let $a$ be a real number and let $b= a + |AB|$. We consider a direct orthonormal basis of $E$, in which the coordinates of $A$ (\emph{resp}. $B$) are $(a, 0)$ (\emph{resp}. $(b, 0)$). The parameters $\r,\rho_1,\rho_2, a, b, \a, \b$ satisfy the conditions of Corollary \ref{Gcurv}; we can thus consider a function $F$ satisfying Corollary \ref{Gcurv}. The graph of $F$ satisfies Proposition \ref{Glemma}.
\begin{flushright}
$\Box$
\end{flushright}

\section*{Acknowledgements}

This work has been initiated during a visit of J-M D. at ICMAT (Madrid); both authors have been partially supported by the grant MTM2014-56350-P. J-M D. also acknowledges the support on the CEFIPRA project 5401 and the ANR-FWF project MuDeRa.\\

\end{document}